\documentclass[11pt,oneside,reqno]{amsart}

\usepackage{graphicx}
\usepackage{tikz-cd}
\usepackage{color}
\usepackage[normalem]{ulem}
\usepackage{mathtools}
\usepackage{amsfonts}
\usepackage{amscd,amsmath}
\usepackage{amssymb}
\usepackage{setspace}
\usepackage{blindtext}
\usepackage{graphicx}
\usepackage{xfrac}
\usepackage{amsmath}

\usepackage{gastex}
\usepackage{longtable}
\usepackage{lscape}
\usepackage{tabularx}
\usepackage{multicol}
\usepackage{verbatim}
\usepackage{multirow}
\usepackage{epsfig,graphics,graphicx}
\usepackage{geometry}
\usepackage{amscd}
\usepackage{xcolor}
\usepackage{hyperref}
\usepackage[utf8]{inputenc}

\usepackage{amssymb,amsmath,amsthm,verbatim,calc}
\numberwithin{equation}{section}
\usepackage[english]{babel}

\definecolor{armygreen}{rgb}{0.29, 0.33, 0.13}
\definecolor{darkgreen}{rgb}{0.0, 0.2, 0.13}

\newtheorem{thm}{Theorem}[section]
\newtheorem{lem}[thm]{Lemma}
\newtheorem{cor}[thm]{Corollary}
\newtheorem{ex}[thm]{Example}

\newtheorem{Def}[thm]{Definition}
\newtheorem{prop}[thm]{Proposition}
\newtheorem{rem}[thm]{Remark}

\newcommand{\Rem}{\begin{rem} \rm}
\newcommand{\bdfn}{\begin{Def} \rm}
\newcommand{\edfn}{\end{Def}}

\newcommand{\ba}{\begin{array}}
\newcommand{\ea}{\end{array}}

\newtheorem{innercustomthm}{Theorem}
\newenvironment{customthm}[1]
{\renewcommand\theinnercustomthm{#1}\innercustomthm}
{\endinnercustomthm}

\begin{document}
\title{On Nonlinear Idempotents and Metric Projections in Normed Spaces}

\author{Himanshu Kumar}
\address{Department of Applied Sciences, Indian Institute of Information Technology Allahabad, Prayagraj-211015, U.P., India.}
\email{himanshumath1507@gmail.com}
        
\author{Abdullah Bin Abu Baker}
\address{Department of Applied Sciences, Indian Institute of Information Technology Allahabad, Prayagraj-211015, U.P., India.}
\email{abdullahmath@gmail.com}

\author{Fernanda Botelho}
\address{Department of Mathematical Sciences, The University of Memphis, Memphis, TN 38152, USA.}
\email{mbotelho@memphis.edu}

		

\thanks{The first-named author gratefully acknowledges the Ministry of Education, New Delhi (India), for financial support, and IIIT Allahabad, India, for providing the resources and infrastructure to carry out this research. The second-named author is partially supported by Anusandhan National Research Foundation (MATRICS) grant No. MTR/2022/000710.}
		
\date{\today}
		
\begin{abstract}
Let $X$ be a normed linear space. A map $P: X \rightarrow X$ is called idempotent if $P^2 = P$. In this paper, we introduce some new classes of idempotent maps (linearity is not assumed) on $X$ and study their properties. A well-known example of an idempotent map is the metric projection $P_K$ onto a Chebyshev subset $K$ of $X$. We provide some necessary conditions for the map $I- P_K$ to be idempotent, where $I$ denotes the identity operator on $X$. 
\end{abstract}
		
\maketitle
		
\thispagestyle{empty}

\section{Introduction and Statement of  Main Results}

Let $X$ be a complex normed linear space. A map $P: X \rightarrow X$ is called idempotent if $P^2 = P$. A bounded (continuous) and idempotent linear operator on $X$ will be referred as projection. Throughout this paper, we follow this terminology, and no linearity is assumed when we consider idempotent maps. Many (simple) properties of projections are not satisfied by idempotents due to the lack of linearity structure in the later class. To illustrate this, we observe that if $P$ is an idempotent map on $X$, then $I-P$ may not be idempotent, where $I$ denotes the identity operator on $X$. A simple example is to define $P: X \rightarrow X$ as $P(x) = \frac{x}{\|x\|}$ when $x \neq 0$ and $P(0) = 0$, see \cite{HAF}. In addition, the absence of linearity makes it more difficult to perform basic algebraic computations on idempotent maps. If $P$, $Q$, and $R$ are idempotent maps on $X$, then $(P + Q)R = PR + QR$, but $P(Q + R) = PQ + PR$ is not necessarily true. Indeed, for fixed nonzero vectors $a, b \in X$, we define $P, Q, R: X \rightarrow X$ as $P(x) = a$ when $x \neq 0$ and $P(0) = 0$, $Q(x) = b$, $R(x) = -b$. Then $P(Q + R) = 0$ and $(PQ + PR)(x) = 2a$. Hence, new techniques are needed to study idempotent maps, and we develop some of them in this paper. A well-known class of idempotent maps is that of metric projections, which we will recall and discuss in detail.

We consider the following new classes of idempotent maps on a normed space $X$. Among these classes are those that are intertwined with surjective (not necessarily linear) isometries supported by the underlying space. Isometries on normed spaces are distance-preserving maps, that is, a map $T: X \rightarrow X$ is called an isometry if $\| Tx - Ty \| = \| x - y \|$ for all $x, y \in X$. 

\begin{Def}
Let $X$ be a normed space. Let $\mathcal{C} = \{P_1, P_2, \ldots, P_k\}$, $k \geq 2$, be a finite collection of nonzero distinct idempotent maps on $X$. We say that $\mathcal{C}$ defines a {\bf \em partition of identity by idempotents} if $\sum_{i=1}^k P_i = I$.
\end{Def}

\begin{Def} Let $\mathcal{C} = \{P_1, P_2, \ldots, P_k\}$, $k \geq 2$, be a finite collection of nonzero distinct idempotent maps on $X$ which defines a partition of the identity by idempotents. Then $\mathcal{C}$ 
\begin{enumerate}
\item is said to be a {\bf \em commuting partition of the identity by idempotents} if $P_iP_j = P_jP_i$ for $i \neq j$, $i, j = 1, 2, \ldots, k$.
		
\item is said to be an {\bf \em orthogonal partition of the identity by idempotents} if $P_iP_j = 0$ for $i \neq j$, $i, j = 1, 2, \ldots, k$.
		
\item defines an {\bf \em idempotent resolution of the identity} if each $P_i$ is bi-potent, $i = 1, 2, \ldots, k$.
\end{enumerate}
\end{Def}

We recall that a map $P$ on a normed space $X$ is said to be bi-potent if both $P$ and $I-P$ are idempotent maps \cite{HAF}. We also recall 

\begin{lem} \cite[Lemma 1.3]{HAF}
Let $P$ be a map on a normed space $X$. If $P$ is idempotent, then $(I - P)P = 0$. Further, $P$ is bi-potent if and only if $P(I-P) = 0$.
\end{lem}

\begin{ex}
	
\begin{enumerate}
\item Let $C[0,1]$ be the space of all real-valued continuous functions on $[0,1]$ with the sup norm $\|.\|_\infty$. Define $P_1, P_2: C[0,1] \rightarrow C[0,1]$ by $P_1f(x) = \max\{f(x),0\}$ and $P_2f(x) = \min\{f(x),0\}$. Then $\mathcal{C} = \{P_1, P_2\}$ is an orthogonal partition of the identity by idempotents. The collection $\mathcal{C}$ also defines an idempotent resolution of the identity.
		
\item Let $E_1, E_2, \ldots, E_k$ be pairwise disjoint subsets of a normed spaces $X$ such that $ X=\bigcup_{i=1}^{k} E_{i}$. For $i = 1,2,\ldots,k$, define $P_i: X \rightarrow X$ by $P_i = I \cdot \chi_{E_i}$, where $\chi_{E_i}$ is the characteristic function of the set $E_i$. Then $\mathcal{C} = \{P_1, P_2, \ldots, P_k\}$ is an orthogonal partition which defines an idempotent resolution of the identity. 
\end{enumerate}

\end{ex}

It is easy to see that if $\mathcal{C} = \{P_1, P_2, \ldots, P_k\}$ is an orthogonal partition of the identity by idempotents, then it is commuting. Our first result shows that the converse holds true.

\begin{customthm}{A} \label{CO}
Let $\mathcal{C} = \{P_1, P_2, \ldots, P_k\}$, $k \geq 2$, be a finite collection of nonzero distinct idempotent maps on a normed space $X$ which defines a partition of the identity by idempotents. If $\mathcal{C}$ is a commuting partition of the identity by idempotents, then it is orthogonal.
\end{customthm}

Let $\mathbb{T}$ denote the unit circle in the complex plane. 

\begin{Def} \label{kCP} Let $X$ be a complex normed space. Let $\mathcal{C} = \{P_1, P_2, \ldots, P_k\}$, $k \geq 2$, be an orthogonal partition which defines an idempotent resolution of the identity. Then $\mathcal{C}$
	
\begin{enumerate}
\item is said to be  a family of {\bf \em \boldmath $k$-circular bi-potents} ($kCB$, for short), if $\sum_{i=1}^k \lambda_i P_i$ is a  surjective isometry on $X$ for all $\lambda_i \in \mathbb{T}$.
		
\item is said to be  a family of {\bf \em generalized \boldmath $k$-circular bi-potents} ($GkCB$, for short), corresponding to a surjective isometry $T$ on $X$ if there exist distinct $\lambda_1, \lambda_2, \ldots, \lambda_k \in \mathbb{T}$ such that $T = \sum_{i=1}^k \lambda_i P_i$.
\end{enumerate}	

Each $P_i$ in $(1)$ and $(2)$ is called a $k$-circular bi-potent and a generalized $k$-circular bi-potent, respectively.
\end{Def}

\begin{Def} \label{kCI} Let $X$ be a complex normed space. Let $\mathcal{C} = \{P_1, P_2, \ldots, P_k\}$, $k \geq 2$, be an orthogonal partition of the identity by idempotents. Then $\mathcal{C}$
	
\begin{enumerate}
\item is said to be  a family of {\bf \em \boldmath $k$-circular idempotents} ($kCI$, for short), if $\sum_{i=1}^k \lambda_i P_i$ is a  surjective isometry on $X$ for all $\lambda_i \in \mathbb{T}$.
		
\item is said to be a family of {\bf \em generalized \boldmath $k$-circular idempotents} ($GkCI$, for short), corresponding to a surjective isometry $T$ on $X$ if there exist distinct $\lambda_1$, $\lambda_2, \ldots, \lambda_k \in \mathbb{T}$ such that $T = \sum_{i=1}^k \lambda_i P_i$.
\end{enumerate}	

Each $P_i$ in $(1)$ and $(2)$ is called a $k$-circular idempotent and a generalized $k$-circular idempotent, respectively. We also say that $\mathcal{C}$ is a family of (generalized) $k$-circular idempotents corresponding to the isometry $T$. We sometimes refer to $T$ as the isometry associated with the family $\mathcal{C}$. 
\end{Def}

Our second result is the following. 

\begin{customthm}{B} \label{P_icont}
Let $\mathcal{C} = \{ P_1, P_2, \ldots, P_k\}$, $k \geq 2$, be a family of $k$-circular idempotents on a normed space $X$. Then each $P_i$ is continuous, $1 \leq i \leq k$.
\end{customthm}

\begin{rem}
	
\begin{enumerate}
\item It is clear that (generalized) $k$-circular bi-potents are (generalized) $k$-circular idempotents. Moreover, for $k = 2$, these two classes coincide, that is, $P$ is a (generalized) $2$-circular bi-potent if and only if it is a (generalized) $2$-circular idempotent. Indeed, if $P_1$ and $P_2$ are idempotent maps such that $P_1 + P_2 = I$, then they are bi-potents.
		
\item $2$-circular idempotents and generalized $2$-circular idempotents will be referred as bi-circular idempotents ($bCI$, for short) and generalized bi-circular idempotents ($GbCI$, for short), respectively.
\end{enumerate}

\end{rem}

In the linear case, there is no distinction between Definitions \ref{kCP} and \ref{kCI}, since every projection is bi-potent. The linear version of the above classes of maps is known as $k$-circular projections and generalized $k$-circular projections; see \cite{AS, A, JI, MH1, DI} and \cite{DCE}. In the last two decades, these classes of projections, especially for the case $k = 2$, have been studied extensively for various Banach spaces. We refer interested readers to the papers \cite{FJ, MDC, JJ, LB1, LB2} and the references therein. 
 
In \cite{FT, FT2}, Botelho and Miura described $GbCI$ on the space of continuously differentiable complex-valued functions on $[0,1]$. The authors of this paper in \cite{HAF} characterized $GbCI$ on the space $\mathcal{S}_A$ of analytic functions on the open unit disc $\mathbb{D}$ whose derivative can be extended to the closed unit disc $\overline{\mathbb{D}}$, and the space $\mathcal{S}^\infty$ of analytic functions on $\mathbb{D}$ with bounded derivatives.	
	
\section{Proofs of Main Results}

First we present the proof of Theorem \ref{CO}. 

\begin{proof}[Proof of Theorem \ref{CO}]
We first note that for $k = 2$, we have $P_1 + P_2 = I$ and $P_1 + P_2P_1 = P_1$. Hence $P_1P_2 = P_2P_1 = 0$.

Let $\mathcal{C} = \{P_1, P_2, \ldots, P_k\}$, $k >2$, be a commuting partition. We will prove that

\begin{equation} \label{orth}
\prod_{i =1}^n P_{\alpha_i} = 0;\ n = 2, 3, \ldots, k, 
\end{equation}
where $\alpha_i \in \{1,2, \ldots, k\}$, and are distinct. 

Since $\mathcal{C}$ defines a partition of the identity, we have
\begin{equation} \label{resol}
\sum_{i=1}^k P_i = I.
\end{equation}

First, we prove Equation \eqref{orth} for $n = k$. 

Multiplying Equation \eqref{resol} by $P_1 P_2 \cdots P_k$, from the commutativity assumption, we get

\begin{align} \label{n=k} 
& \sum_{i=1}^k P_i P_1 P_2 \cdots P_k  = P_1 P_2 \cdots P_k, \nonumber \\ 
& \Longrightarrow \sum_{i=1}^k  P_1 P_2 \cdots P_k  = P_1 P_2 \cdots P_k, \nonumber \\
&\Longrightarrow (k - 1) P_1 P_2 \cdots P_k  = 0, \nonumber \\
&\Longrightarrow \prod_{i =1}^k P_i  = 0, \ \ \ \mbox{(since} \ k\neq 1).
\end{align} 

For the case $n = k - 1$, we assume without loss of generality that $\alpha_i \neq k$, for all $i = 1, \ldots, k - 1$. That is, we need to show  

\begin{equation*}
\prod_{i =1}^{k - 1} P_i = 0.
\end{equation*}

At this point, we multiply Equation \eqref{resol} by $P_1 P_2 \cdots P_{k - 1}$ to get

\begin{equation*}
\sum_{i=1}^k P_i P_1 P_2 \cdots P_{k-1}  = P_1 P_2 \cdots P_{k-1}.    
\end{equation*}

This implies that

\begin{equation*}
(k - 1) \prod_{i =1}^{k - 1} P_i + \prod_{i =1}^k P_i = \prod_{i =1}^{k - 1} P_i.
\end{equation*}
Using Equation \eqref{n=k}, we conclude
\begin{equation*}
\prod_{i =1}^{k - 1} P_i = 0, \quad (\mbox{since} \ k \neq 2).
\end{equation*}

Continuing this process, suppose that we have proved Equation \eqref{orth} for $n = m + 1$. To prove the result for $n = m$, we use the following steps.

{\bf \boldmath Steps to prove Equation \eqref{orth} for $n = m$, $2 \leq m \leq k - 2$}.
\begin{enumerate}
\item Multiplying Equation \eqref{resol} by $P_{\alpha_1} P_{\alpha_2} \cdots P_{\alpha_m}$, we get the following equation

\begin{equation*}
\sum_{i=1}^k P_i P_{\alpha_1} P_{\alpha_2} \cdots P_{\alpha_m} = P_{\alpha_1} P_{\alpha_2} \cdots P_{\alpha_m}.    
\end{equation*}

\item Separating those terms in the above equation in which $i = \alpha_j$, $j= 1, \ldots, m$, from others, we have 
\begin{equation*}
\sum_{i = \alpha_1, \ldots, \alpha_m} P_i P_{\alpha_1} P_{\alpha_2} \cdots P_{\alpha_m} + \sum_{i \neq \alpha_1, \ldots, \alpha_m} P_i P_{\alpha_1} P_{\alpha_2} \cdots P_{\alpha_m} = P_{\alpha_1} P_{\alpha_2} \cdots P_{\alpha_m}. 
\end{equation*}

\item Using commutativity of $P_i$s, we conclude 

\begin{equation*}
m \ P_{\alpha_1} P_{\alpha_2} \cdots P_{\alpha_m} + \sum_{i \neq \alpha_1, \ldots, \alpha_m} P_i P_{\alpha_1} P_{\alpha_2} \cdots P_{\alpha_m} = P_{\alpha_1} P_{\alpha_2} \cdots P_{\alpha_m}. 
\end{equation*}

\item Since the product of an arbitrary $(m+1)$-idempotents from $\mathcal{C}$ is $0$, the second term of the above equation vanishes. It follows that 

\begin{equation*}
(m-1) \prod_{i =1}^m P_{\alpha_i} = 0, \text{ or } \prod_{i =1}^m P_{\alpha_i} = 0.
\end{equation*}
\end{enumerate}

This completes the proof of Theorem \ref{CO}.
\end{proof}

Before proceeding to the proof of Theorem \ref{P_icont}, we mention some properties of $GkCI$ which generalize some of the results proved for $GbCI$ in \cite{HAF}.
   
\begin{lem} \label{TP=lambdaP}
Let $\mathcal{C} = \{P_1, P_2, \ldots, P_k\}$, $k \geq 2$, be a family of generalized $k$-circular idempotents corresponding to a surjective isometry $T$ such that $T = \sum_{i=1}^k \lambda_i P_i$. Then $TP_i = \lambda_i P_i$ for all $i = 1, 2, \ldots, k$.
\end{lem}

\begin{proof}
Since $T = \sum_{i=1}^k \lambda_i P_i$, multiply this equation by $P_i$ and use the orthogonality of $P_i$s to get the desired result. 
\end{proof} 

Consider a family $\mathcal{C} = \{P_1, P_2, \ldots, P_k\}$ of $GkCI$ corresponding to a surjective isometry $T$. Mazur-Ulam theorem states that some translation of $T$ is a real linear isometry. In other words, $T = T(0) + S$, where $S$ is a real linear isometry on $X$. The next proposition shows that $T$ maps the origin to itself.

\begin{prop} \label{T0=0}
Let $\mathcal{C} = \{P_1, P_2, \ldots, P_k\}$, $k \geq 2$, be a family of generalized $k$-circular idempotents corresponding to a surjective isometry $T$. Then $T(0) = 0$. Moreover, $T$ is real linear.
\end{prop}

\begin{proof}
Let $T = \sum_{i=1}^k \lambda_i P_i$, where $\lambda_i$ are distinct modulus one complex numbers. By Mazur-Ulam theorem, $T = T(0) + S$, where $S$ is a real linear surjective isometry. By Lemma \ref{TP=lambdaP}, we have $TP_i = \lambda_i P_i$ for $i = 1, 2, \ldots, k$. It follows that $T(0) +  SP_i = \lambda_i P_i$ for $i = 1, 2, \ldots, k$. Adding these $k$ equations, we obtain $k T(0) +  \sum_{i=1}^k S P_i =  \sum_{i=1}^k \lambda_i P_i$. Since $S$ is real linear and the collection $\mathcal{C}$ defines a partition of the identity, we conclude that $kT(0) + S = T(0) + S$. Hence, $T(0) = 0$. This also shows that $T$ is real linear.
\end{proof}

\begin{proof}[Proof of Theorem \ref{P_icont}]
Let $\mathcal{C} = \{ P_1, P_2, \ldots, P_k\}$, $k \geq 2$, be a family of $k$-circular idempotents on $X$.	Then for each $i = 1, 2, \ldots, k-1$ and $\lambda_1^i, \lambda_2^i, \ldots, \lambda_k^i \in \mathbb{T}$, 
$$
T_i = \lambda_1^i P_1 + \lambda_2^i P_2 + \cdots + \lambda_k^i P_k
$$
is a surjective isometry on $X$. Moreover, $I = P_1 + P_2 + \cdots +P_k$. These equations can be described by the matrix equation $M \widetilde{\bf{P}} = \widetilde{\bf{T}}$, where
$$
M = 
\begin{bmatrix}
1 & 1 & \ldots & 1 \\
\lambda_1 & \lambda_2 & \ldots & \lambda_k \\
\vdots &  \vdots & \vdots & \vdots \\
\lambda_1^{k-1} & \lambda_2^{k-1} & \ldots & \lambda_k^{k-1} 
\end{bmatrix},\
\widetilde{\bf{P}} = 
\begin{bmatrix}
P_1 \\
P_2 \\
\vdots \\
P_k
\end{bmatrix} 
\text{ and }
\widetilde{\bf{T}} 
= \begin{bmatrix}
I \\
T_1 \\
\vdots \\
T_{k-1}
\end{bmatrix}.
$$
We observe that $M$ is a Vandermonde matrix of order $k$ and 
$$
\text{det } M = \displaystyle \prod_{i > j} (\lambda_i - \lambda_j) \neq 0.
$$ 
This implies that $\widetilde{\bf{P}} = M^{-1} \widetilde{\bf{T}}$. Since each isometry $T_i$ is continuous, we conclude that each $P_i$ is continuous. 
\end{proof}

The following corollary is immediate from Theorem \ref{P_icont} and Proposition \ref{T0=0}. 

\begin{cor} \label{Real-kCI}
Let $\mathcal{C} = \{ P_1, P_2, \ldots, P_k\}$, $k \geq 2$, be a family of $k$-circular idempotents on a normed space $X$. Then each $ P_i$ is real linear, $i = 1,2, \ldots, k$.
\end{cor}

\section{Considerations on Metric Projections}

In this section, we consider an important class of idempotent maps called metric projections and see when it is bi-potent.

\begin{Def}
Let $K$ be a non-empty subset of a normed linear space $X$ and let $x \in X$. The metric projection of $X$ onto $K$ is the set-valued map defined by
$$
P_K(x) = \{y \in K: ||x - y|| = dist (x, K)\}, 
$$
where $dist (x, K) = \inf_{y \in K} ||x - y||$. The set $K$ is called proximinal (resp. Chebyshev) if $P_K(x)$ contains at least (resp. exactly) one point for every $x \in X$.  
\end{Def}

It is known that proximinal sets are closed. If $X$ is a uniformly convex Banach space, then any closed and convex subset of $X$ is Chebyshev. If $K$ is a Chebyshev set in $X$, then $P_K$ is single-valued with range $K$. Moreover, if $x \in K$, then $P_K(x) = x$. Thus, $P_K$ is idempotent. For more on this topic, we refer the reader to the book by Deutsch \cite{D}.

Now, we present some necessary conditions for the bi-potency of metric projections. 

\begin{thm}
Let $K$ be a Chebyshev set in a normed space $X$, and let $P_K$ be a bi-potent map. Then the following statements are true.
\begin{enumerate}
\item $0 \in K$.

\item For every $x \in K^c$, $x - P_K(x) \in K^c$, where $K^c$ denotes the complement of $K$.

\item $0 \notin int(K)$.

\item $0 \in Bdr(K)$, the boundary of $K$.

\item If $K$ is complete and convex, then it is not totally bounded.
\end{enumerate}
\end{thm}

\begin{proof}
\begin{enumerate}
\item Since $P_K(I - P_K) = 0$ and the range $P_K$ is $K$.

\item Let $x \in K^c$ such that $x - P_K(x) \in K$. Then $P_K (I-P_K)(x)= x - P_K(x) \neq 0$. Thus, $I-P_K$ is not idempotent, a contradiction.

\item Suppose $0 \in int(K)$, then there exists $\epsilon > 0$ such that $B_\epsilon (0)\subset K$. Let $y \in Bdr(K)$ and $x \in B_\epsilon (y) \setminus K$. Then $\|x-P_K (x)\| \leq ||x - y|| < \epsilon$. It follows that $x - P_K(x) \in K$, a contradiction.

\item This is immediate, since $K$ is closed and $0 \notin int(K)$. 

\item On the contrary, we assume that $K$ is totally bounded. Select $z\in K$ such that $\|z\|=  \max \{\|w\|: w \in K\}$. It is clear that $z \in Bdr(K)$. Moreover, $0 \in Bdr(K)$. Since $K$ is convex, the line segment joining $0$ and $z$ is contained in $K$. Consider a positive $\epsilon < 1$. Then $\epsilon z \in K$ and $(1+\epsilon) z \in K^c$. We claim that $P_K((1+\epsilon) z) = z$. If not, let $z_1=P_K((1+\epsilon) z)$ with $\|z_1-(1+\epsilon) z\|<\epsilon\|z\|$. Then
\begin{align*}
\|z_1\| = \|z_1-(1+\epsilon) z+(1+\epsilon) z\| & \geq (1+\epsilon) \|z\|-\|z_1-(1+\epsilon) z\| \\ 
& > (1+\epsilon) \|z\|-\epsilon\|z\|=\|z\|.
\end{align*}
This is not possible due to the choice of $z$. Therefore, $P_K((1+\epsilon)z) = z$ and $(1+\epsilon)z - P_K((1+\epsilon)z) = (1+\epsilon)z - z = \epsilon z \in K$, a contradiction due to the assertion $(2)$. Hence, $K$ is not totally bounded.
\end{enumerate}
\end{proof}

\textbf{Conflict of Interest:} The authors declare that they have no conflict of interest.

\end{document}